\documentclass[11pt]{article}
\usepackage{amsmath, amsthm, amssymb, setspace, graphics}
\usepackage[all]{xy}
\usepackage{url}

\newcommand{\partentry}[1]{\addtocontents{toc}
                        {\small\bfseries#1\hfill\thepage\par}}

\makeatletter
\def\@part[#1]#2{%
    \ifnum \c@secnumdepth >\m@ne
      \refstepcounter{part}
      \partentry{\protect\makebox[2em][l]{\thepart}#1}
    \else
      \partentry{#1}
    \fi
    {\parindent \z@ \raggedright
     \interlinepenalty \@M
     \normalfont\Large\bfseries\thepart\hspace{1em}#2%
     \markboth{}{}\par}%
    \nobreak
    \vskip 3ex
    \@afterheading}
\def\@spart#1{%
    {\parindent \z@ \raggedright
     \interlinepenalty \@M
     \normalfont\Large\bfseries #1\par}
     \nobreak
     \vskip 3ex
     \@afterheading}
\makeatother
\numberwithin{equation}{subsection}
\newtheorem{thm}[equation]{Theorem}

\newtheorem{pro}[equation]{Proposition}

\theoremstyle{remark}
\newtheorem{defn}[equation]
               {Definition}
\newtheorem{rem}[equation]
               {Remark}

\newcommand{\CC}{\mathbb C}

\newcommand{\HH}{\mathcal H}

\newcommand{\QQ}{\mathbb Q}

\newcommand{\ZZ}{\mathbb Z}

\newcommand{\Ocal}{\mathcal O}

\newcommand{\Lcal}{\mathcal L}

\newcommand{\Dcal}{{\mathcal D}}

\newcommand{\id}{\operatorname{id}}
\newcommand{\Aut}{\operatorname{Aut}}

\newcommand{\pt}{\operatorname{pt}}

\newcommand{\Ext}{\operatorname{Ext}}

\newcommand{\Spec}{\operatorname{Spec}}

\newcommand{\Res}{\operatorname{Res}}
\newcommand{\al}{\alpha}

\newcommand{\fie}{\varphi}
\newcommand{\ka}{\kappa}

\newcommand{\ev}{{\rm ev}}
\newcommand{\ie}{{\it i.e.}}

\newcommand{\cw}{{c}_{W}}

\newcommand{\irr}{\operatorname{irr}}

\newcommand{\ch}{{\rm ch}}
\newcommand{\td}{\operatorname{td}}

\newcommand{\textsum}{{\textstyle{\sum}}}

\providecommand{\abs}[1]{\lvert#1\rvert}
\def\pmmu{{\pmb \mu}}

\def\wt{\widetilde}

\def\d{\partial}
\newcommand{\MMM}{\mathcal M}

\newcommand{\MMMbar}{{\overline{\mathcal M}\hspace{1pt}}}
\newcommand{\CCCbar}{{\overline{\mathcal C}\hspace{1pt}}}
\begin{document}

\title{\textbf{Twisted Gromov--Witten $r$-spin potential
and  Givental's quantization}}
\author{A.~Chiodo, D.~Zvonkine} \maketitle

\begin{abstract}
The universal curve $\pi\colon\CCCbar \rightarrow \MMMbar$ over the moduli
space~$\MMMbar$ of stable $r$-spin maps to a target K\"ahler manifold~$X$
carries a universal spinor bundle $\Lcal \to \CCCbar$.
Therefore the moduli space $\MMMbar$
itself carries a natural K-theory class $R \pi_* \Lcal$.

We introduce a {\em twisted} $r$-spin Gromov--Witten potential
of~$X$ enriched with Chern characters of $R \pi_* \Lcal$. We show
that the twisted potential can be reconstructed from the
ordinary $r$-spin Gromov--Witten potential of~$X$ via an
operator that assumes a particularly simple form in
Givental's quantization formalism.
\end{abstract}

\section{Introduction}

In~\cite{Mumford} D.~Mumford used the Grothendieck--Riemann--Roch formula
to express the Chern characters of the Hodge bundle over the
moduli space of stable curves via other tautological classes.

C.~Faber and R.~Pandharipande~\cite{FP} generalized his result on
spaces of stable maps. They showed that the resulting formula allows
one to express the so-called {\em twisted} Gromov--Witten
potential (enriched with Chern characters of the Hodge bundle)
of any target K\"{a}hler manifold via the usual Gromov--Witten
potential.

A.~Givental~\cite{Givental} noted that the above result admitted a
strikingly concise formulation in the framework of his quantization
formalism. This allows many explicit calculations of twisted
Gromov--Witten potentials starting from known ``untwisted'' potentials
(see \cite{CG}, \cite{Ts}, \cite{CCIT} and
references therein).

In the present paper we generalize all these steps to the
spaces of $r$-spin structures and maps,
Theorem~\ref{Thm:main}. Witten's $r$-spin conjecture,
proved in~\cite{FSZ}, determines the $r$-spin untwisted Gromov--Witten potential of the point. Theorem~\ref{Thm:main} can be regarded as the natural tool to calculate $r$-spin Gromov--Witten potentials beyond the untwisted cases,
see Remark~\ref{rem:main}.

\subsection{Mumford's formula and Givental's formalism}

\subsubsection{Moduli spaces}
Let $\MMMbar_{g,n}$ denote the moduli space of
genus-$g$ stable curves with $n$ marked points, and let
$\CCCbar_{g,n} \rightarrow \MMMbar_{g,n}$ be the universal curve.
In one-to-one correspondence with the marked points,
natural cohomology classes
$\psi_1, \dots, \psi_n \in H^2(\MMMbar_{g,n},\QQ)$
are defined on the moduli space:
the cotangent lines at the $i$th markings form a line bundle
$L_i\to \MMMbar_{g,n}$, and $\psi_i$ equals $c_1(L_i)$.
Moreover, if
$p\colon\MMMbar_{g,n+1} \rightarrow \MMMbar_{g,n}$ is the forgetful map,
we define $\ka_d  = p_*(\psi_{n+1}^{d+1}) \in
H^{2d}(\MMMbar_{g,n},\QQ)$, for $d \geq 1$.
(We refer to~\cite{GraPan} for an overview of these and other
geometrically defined cohomology classes on $\MMMbar_{g,n}$ and their properties.)

\subsubsection{Boundary}
Set
$$
\Dcal =
\MMMbar_{g-1,n+2} \; \sqcup \!\!\!\!
\bigsqcup_{
\substack{
l + l' = g \\ I \sqcup I' = \{ 1, \dots, n \}
}} \!\!\!\!
\MMMbar_{l, I \cup \{n+1 \}} \times \MMMbar_{l', I' \cup \{n+2 \}}.
$$
Then there is a natural map $j\colon \Dcal \rightarrow
\MMMbar_{g,n}$, obtained by gluing together the marked points
number $n+1$ and $n+2$. The image of~$j$ is the boundary of
$\MMMbar_{g,n}$.

\subsubsection{Mumford's formula}
In~\cite{Mumford}, D.~Mumford applied the Grothendieck--Riemann--Roch
formula to express the Chern characters $\ch_d(\Lambda)$ of the
Hodge bundle $\Lambda \rightarrow \MMMbar_{g,n}$.
His formula reads, for $d \geq 1$,
\begin{equation} \label{Eq:Mumford}
\ch_d(\Lambda) = \frac{B_{d+1}}{(d+1)!}
\left[ \kappa_d - \sum_{i=1}^n \psi_i^d
+ \frac12 j_*
\biggl(\,
\sum_{a=0}^{d-1} \psi_{n+1}^a (-\psi_{n+2})^{d-1-a}
\biggr)
\right],
\end{equation}
where $B_{d+1}$ is the $(d+1)$st Bernoulli number.
In particular, $\ch_d(\Lambda)=0$ for even~$d$, because odd Bernoulli
numbers $B_{d+1}$ vanish.

\subsubsection{Differential operators}
Let $X$ be a K\"{a}hler manifold and $(h_\mu)_{\mu \in B}$
a basis of the cohomology
space $H=H^*(X,\QQ)$ such that $h_1 = 1$. For simplicity we
assume that $X$ has only even cohomology.
Let $g_{\mu\mu'}$ be the matrix of the Poincar\'{e}
duality form on $H$ in the basis~$(h_\mu)$.
Denote by $X_{g,n,D}$ the space of degree~$D$ stable maps
with target~$X$, of genus~$g$, with $n$ marked points.
Here $D$ is an effective cycle in~$H_2(X,\ZZ)$.
Let $[X_{g,n,D}]^v$ denote its virtual fundamental class.

Introduce the following generating series in the variables
Q, $s_d$, $d \geq 1$, and $t_a^\mu$, $a \geq 0$, $\mu \in B$:
$$
F_g({\pmb t},{\pmb s})= \sum_{n \geq 1} \sum_D
Q^D \sum_{
\substack{a_1, \dots, a_n \\ \mu_1, \dots, \mu_n}
} \; \;
\int\limits_{[X_{g,n,D}]^v} \!\!\!\!\!\!
\exp\Bigl(\sum_{d \geq 1} s_d \, \ch_d(\Lambda)\Bigr)
\; \prod_{i=1}^n \psi_i^{a_i} \ev_i^*(h_{\mu_i})  \cdot
 \frac{t_{a_1}^{\mu_1} \dots t_{a_n}^{\mu_n}}{n!}.
$$
The series $F= \sum_{g \geq 0} \hbar^{g-1} F_g$
is called the
{\em twisted Gromov--Witten potential of~$X$}.
Let $Z = \exp\,F$.

Denote by $g^{\mu\mu'}$ the inverse matrix of $g_{\mu\mu'}$.
From the expression of $\ch_d(\Lambda)$,
C.~Faber and R.~Pandharipande~\cite{FP} deduced the
following claim: {\em we have
$$
\frac{\d Z}{\d s_d} = L_d Z,
$$
where $L_d$ is the linear differential operator depending only on the
variables $t_{a}^\mu$:
$$
L_d = \frac{B_{d+1}}{(d+1)!} \left[
\frac{\d}{\d t_{d+1}^1} -
\sum_{a \geq 0,\mu} t_a^\mu \frac{\d}{\d t_{a+d}^\mu}
+\frac{\hbar}2
\sum_{
\substack{a+a' = d-1 \\ \mu,\mu'}
} (-1)^{a'} g^{\mu\mu'} \frac{\d^2}{\d t_a^\mu \d t_{a'}^{\mu'}}
\right].
$$
}

\subsubsection{Givental's quantization} \label{Sssec:Givental}
Let $\HH = H((z^{-1}))$ be the space of $H$-valued Laurent series in~$z$,
finite in the positive direction and possibly infinite in
the negative direction. This space bears a natural symplectic
form
$$
\omega(f_1,f_2) = \Res_{z=0} \sum_{\mu,\mu'} g_{\mu\mu'}
f_1^\mu(-z) f_2^{\mu'}(z).
$$
If we write a Laurent series in the form
$$
f(z) = \sum_{\substack{a \geq 0\\ \mu \in B}} q_a^\mu z^a h_\mu +
\sum_{\substack{a \geq 0\\ \mu,\mu' \in B}}
p_{a,\mu}g^{\mu\mu'} (-1/z)^{a+1}  h_{\mu'}
$$
(the first sum is finite, while the second one can be infinite),
then $p_{a,\mu}, q_a^\mu$ form a set of Darboux coordinates on~$\HH$.
The coordinates $q_a^\mu$ are identified with the variables
$t_a^\mu$ above via
\begin{eqnarray*}
q_1^1 &=& t_1^1-1, \\
q_a^\mu &=& t_a^\mu \qquad \mbox{for other } a \mbox{ and } \mu.
\end{eqnarray*}
The strange-looking shift for $a=1, \mu=1$ is called the {\em dilaton shift}.

Functions (or {\em Hamiltonians}) on the
space $\HH$ are quantized according to Weyl's
rules: $q_a^\mu$ is transformed into the operator of multiplication
by $q_a^\mu/\sqrt{\hbar}$, while $p_{a,\mu}$ is transformed into the
operator $\sqrt{\hbar} \, \d/\d q_a^\mu$. By convention, the derivations
are applied before the multiplications. Using these rules,
we obtain the following result.

{\em The operator $L_d$ is the Weyl quantization of
the Hamiltonian
$$
P_d = \frac{B_{d+1}}{(d+1)!} \left[
- \sum_{\substack{a \geq 0\\ \mu \in B}} q_{a}^\mu p_{a+d,\mu}
+ \frac12 \sum_{\substack{a+a' = d-1\\ \mu,\mu' \in B}}
(-1)^{a'} g^{\mu\mu'} p_{a,\mu} p_{a',\mu'}
\right].
$$
}

Every Hamiltonian~$P$ determines a vector field on (or an
infinitesimal symplectic transformation of)~$\HH$,
given by~$\omega^{-1}(dP)$.
Since our Hamiltonian is of pure degree~2, the corresponding vector
field is linear. Following~\cite{Givental}, we obtain that {\em the
vector field on $\HH$ determined by $P_d$ is given by the
multiplication by
\begin{equation} \label{Eq:field}
\frac{B_{d+1}}{(d+1)!} z^d.
\end{equation}
}

The Hamiltonian can, of course, be recovered from the
vector field (up to an additive constant).
Thus the not-so-simple Mumford formula~\eqref{Eq:Mumford}
turns out to be encoded in the strikingly simple expression~\eqref{Eq:field}.

\subsubsection{Characteristic classes}
\label{Sssec:charclasses}

Moduli spaces of $r$-spin curves $\MMMbar_{g,n}^{r,\pmb m}$
and of $r$-spin maps $X_{g,n,D}^{r,\pmb m}$ will be
introduced in Section~\ref{Sec:rspin}.
Essentially they classify stable
curves or maps enriched with an
$r$-spin structure in the sense of Witten:
in the case of a smooth curve with no markings,
an $r$-spin structure is a line bundle $\Lcal$, which is an
$r$th tensor root of the canonical line bundle on the curve.

These moduli spaces come with universal curves
$\pi: \CCCbar_{g,n}^{r,\pmb m} \to \MMMbar_{g,n}^{r,\pmb m}$ and
$\pi: \CCCbar_{g,n,D}^{r,\pmb m}(X) \to X_{g,n,D}^{r,\pmb m}$.
The universal curve has $n$ sections
$s_1, \dots s_n$ specifying the marked points and carries a
universal $r$-spin structure $\Lcal_{g,n}^{r,\pmb m}$
(which is either a line bundle or a sheaf, depending on
the construction – see Section~\ref{Sec:rspin}).
Let $\omega$ be the relative dualizing sheaf of the
universal curve and $\omega_{\log}$ the relative
dualizing sheaf twisted by the divisor
of the sections $\sum [s_i]$.

On the moduli space $\MMMbar_{g,n}^{r,\pmb m}$ or
$X_{g,n,D}^{r,\pmb m}$
we consider the following cohomology classes:
\begin{align*}
&\psi_i=c_1(s_i^*\omega)
,\\
&\kappa_d=\pi_*(c_1(\omega_{\log})^{d+1})
,\\
&\ch_d=\ch_d(R\pi_*\Lcal_{g,n}^{r,\pmb m})
, 
\end{align*}
where $\ch_d$ denotes the term in degree
$d$ of the Chern character.



\begin{rem} The classes $\psi_i$ and $\kappa_d$ are
the pullbacks of the analogous classes in the Chow ring of the
stack $\MMMbar_{g,n}$ of stable curves (respectively, the
stack $X_{g,n,D}$ of stable maps) via the natural morphism
$\MMMbar_{g,n}^{r,\pmb m}\to \MMMbar_{g,n}$
(respectively $X_{g,n,D}^{r,\pmb m} \to X_{g,n,D}$).
\end{rem}



\subsection{The main result}
\label{Ssec:mainresult}

Let $X$ be a target K\"{a}hler manifold and $(h_\mu)$ a basis
in $H^*(X, \QQ)$ (as before, we assume that $X$ has only even
cohomology).

Denote by $\big[ X_{g,n,D}^{r,\pmb m} \big]^v$
the virtual $r$-spin class of $X_{g,n,D}^{r, \pmb m}$
(see Section~\ref{Sec:moduli} for more details). Let
$\ev_i\colon X_{g,n,D}^{r, \pmb m} \rightarrow X$ be
the evaluation maps.

Let $B_d(x)$ be the Bernoulli polynomials. They are defined by
$$
\frac{e^{tx} \, t}{e^t-1} = \sum \frac{B_d(x)}{d!} t^d.
$$

\begin{defn}\label{defn:twisted_rspin}
The power series
$$
F_{g,r}({\pmb t},{\pmb s})=\\
\sum_D Q^D \sum_{n\ge 1}\;\; \frac{1}{n!} \,
\sum_{
\substack{m_1, \dots, m_n \\ a_1, \dots, a_n\\ \mu_1, \dots, \mu_n}
} \; \frac{1}{r^{g-1}}\!\!\!
\int\limits_{\big[X_{g,n,d}^{r,\pmb m}\big]^v} \!\!\!\!
\exp\biggl(\sum_{d\ge 1} s_d \, \ch_d \biggr) \;
\prod_{i=1}^n \psi_i^{a_i} \, \ev_i^*(h_{\mu_i}) \;
\prod_{i=1}^n t_{a_i}^{\mu_i \otimes m_i}.
$$
is called the {\em genus-$g$ $r$-spin twisted Gromov--Witten potential}
of~$X$.
\end{defn}

Let
$F_r({\pmb t},{\pmb s})=
\sum\limits_{g\ge 0}\hbar^{g-1}F_{g,r}({\pmb t},{\pmb s})$
and $Z_r=\exp(F_r).$

Consider the vector space $\QQ^{r-1}$ with basis $e_1, \dots, e_{r-1}$
endowed with the quadratic form
$g$ with coefficients $g_{ab} = \delta_{a+b,r}$.
A diagonal matrix
in basis $e_1, \dots, e_{r-1}$ will be denoted by
$\mbox{diag}[u_1, \dots, u_{r-1}]$.
Denote by $H$ the space $H = H^*(X,\QQ) \otimes \QQ^{r-1}$.
The metric on~$H$ is
the tensor product of the Poincar\'{e} paring on $H^*(X,\QQ)$ and the
metric~$g_{ab}$ on $\QQ^{r-1}$.

Let $\HH = H((\frac1z))$ as in \S\ref{Sssec:Givental}.

\begin{thm} \label{Thm:main}
We have
$$
\frac{\partial Z_r}{\partial s_d} = L_d Z_r,
$$
where $L_d$ is a differential operator in the variables $\pmb t$
obtained by Weyl quantization of the infinitesimal symplectic
transformation of~$\HH$ given by
$$
\frac{z^d}{(d+1)!} \cdot  {\rm Id} \otimes
\mbox{\rm diag} \left[
B_{d+1} \left( \frac1r \right), \dots, B_{d+1} \left( \frac{r-1}r \right)
\right].
$$
\end{thm}

\begin{rem}\label{rem:main}
This theorem actually gives a way of computing the twisted
$r$-spin potential of~$X$ starting from the ordinary
$r$-spin potential of~$X$, not twisted by
the classes $\ch_d$.

More precisely, let $f_r = F_r|_{s_1 = s_2 = \dots = 0}$ be
the ordinary $r$-spin Gromov--Witten potential of~$X$. Then we have
$$
\exp(F_r) = \exp\left(\sum s_d L_d\right) \exp(f_r).
$$
\end{rem}

\subsection{Plan of the paper}
In Section~\ref{Sec:rspin} we give an overview of the
existing constructions of the natural compactification of
spaces of $r$-spin curves.
In Section~\ref{Sec:moduli} we explain in more detail
the properties of the moduli spaces of $r$-stable spin
maps and the virtual fundamental classes.
In Section~\ref{Sec:GRR} we recall and slightly generalize the results
obtained in~\cite{ch} by applying the Grothendieck--Riemann--Roch
formula to the spinor bundle on the so-called universal $r$-bubbled curve.
In Section~\ref{Sec:quant}
we put these results in the framework of Givental's quantization.

\section{Moduli of $r$-spin curves: an overview}
\label{Sec:rspin}

Here we describe the construction of the moduli
spaces of $r$-stable spin curves and maps.
The integer $r \geq 2$ is fixed once and for all.
We denote by $\ZZ_r$ the group of $r$th roots of unity
and set $\xi_r = \exp(2 \pi i/r)$.

All schemes and stacks throughout this paper are defined over~$\CC$.

\subsection{Smooth curves}
Let $(C;s_1, \dots, s_n)$ be a smooth genus-$g$ curve with
$n \geq 1$ marked points. Choose $n$ integers
$m_1, \dots, m_n \in \ZZ$ such that
$2g-2+n - \sum m_i$ is divisible by~$r$. An {\em $r$-spin
structure of type $\pmb m = (m_1, \dots, m_n)$} on~$C$ is a line
bundle $\Lcal$ over~$C$ together with an identification
$$
\fie\colon \Lcal^{\otimes r} \xrightarrow{ \ \sim\ }
\omega_{\log}(-\textsum_{i=1}^n m_i[s_i]),
$$
where $\omega$ is the cotangent line bundle of~$C$
and $\omega_{\log} = \omega(\sum [s_i])$.

The number $m_i$ is called the {\em index} of $\Lcal$ at~$x_i$.

There are exactly $r^{2g}$ nonisomorphic $r$-spin structures
of type~$\pmb m$ on every smooth curve. Each of them has~$r$
``trivial'' automorphisms $\Lcal \rightarrow \Lcal$
given by the multiplications by $r$th roots of unity
along its fibers. A curve endowed with an $r$-spin structure is called
a {\em smooth $r$-spin curve}.

The moduli space $\MMM_{g,n}^{r,\pmb m}$ of smooth
$r$-spin curves is an $r^{2g}$-sheeted unramified covering
of the moduli space of smooth curves $\MMM_{g,n}$ in the sense
that its fibre is constant and consists of $r^{2g}$ copies of the same
$0$-dimensional stack. However,
since each point of the fibre is equipped with the
$r$ ``trivial'' automorphisms mentioned above,
the forgetful map $\MMM_{g,n}^{r,\pmb m} \rightarrow \MMM_{g,n}$
is of degree~$r^{2g-1}$.

\begin{rem} \label{rem:shift}
Let $\pmb m = (m_1, \dots, m_n)$ and
$\pmb{m'} = (m_1, \dots, m_i+r, \dots, m_n)$.
There is a natural isomorphism between $\MMM_{g,n}^{r, \pmb m}$
and $\MMM_{g,n}^{r, \pmb{m'}}$. Therefore, from now on, we will always
choose $m_1, \dots, m_n$ in $\{1, \dots, r\}$.
\end{rem}

\subsection{Curves with nodes}
\label{Ssec:3curves}

There exists a natural compactification
$\MMMbar_{g,n}^{r, \pmb m}$ of the space
$\MMM_{g,n}^{r, \pmb m}$. This
compactification can be constructed in three different (but equivalent)
ways, and there are, accordingly, three different versions
of the universal curve over the compactified moduli space
(see Figure~\ref{Fig:3versions}).
We are going to describe the behavior of the universal
curve at the neighborhood of a node in all three versions.

In the universal curve $\pi:\CCCbar_{g,n} \rightarrow \MMMbar_{g,n}$
over the space of stable curves, from a local point of view,
there is a unique type of node.
The local picture of~$\pi$ at the neighborhood of a node is given by
$(x,y) \mapsto t = xy$. However, in the universal curve
over the space of $r$-stable curves, there are $r$ different
types of nodes: they are distinguished by assigning
to the branches of the curve at the node two integers
$a,b \in\{1, \dots, r \}$ such that either $a=b=r$
or $a+b=r$, see below.

\paragraph{1) coarse $r$-spin curves.} Coarse (or scheme-theoretic)
$r$-spin curves were introduced by Jarvis~\cite{Ja_comp,Ja_geom} using
relatively torsion-free sheaves.

The local picture of the universal coarse $r$-spin curve at
the neighborhood of
a node is given by the equation $xy=t^r$, where $t$ is
a local coordinate on the moduli space. Thus the universal
curve has an $A_{r-1}$ singularity at $x=y=t=0$.

If $a=b=r$, then locally, at the neighborhood of
the node, $\Lcal$ is a line bundle endowed with an isomorphism
$$
\Lcal^{\otimes r} \xrightarrow{ \ \sim\ }
\omega_{\log}.
$$
(There is no need to twist $\omega_{\log}$ by a divisor
of marked points since we are working at the neighborhood
of the node.)

If $a+b=r$, then $\Lcal$ is no longer a line bundle, but a
(rank-$1$ torsion-free) sheaf generated by two elements
$\eta$ and $\xi$, modulo the relations $y\xi = t^a \eta$,
$x\eta = t^b \xi$. There is a map from $\Lcal^{\otimes r}$ to
the sheaf of sections of $\omega_{\log}$ given by
$$
\xi^i \eta^j \mapsto t^{jb} x^{a-j} \frac{dx}{x}
=  - t^{ia} y^{b-i} \frac{dy}{y},
$$
for $i+j=r$. This map is, of course, not an isomorphism.

A stable curve endowed with a sheaf~$\Lcal$ like that will be called
a {\em stable coarse $r$-spin curve}.
We often need to generalize this to stable maps from $C$ to a K\"{a}hler
manifold $X$. In this case the source curve is a prestable curve,
which is not
necessarily stable. When $C$ is equipped with a sheaf~$\Lcal$ as above, we say
that $C$ is a {\em prestable coarse $r$-spin curve}.

\begin{figure}
\begin{center}
\
\includegraphics{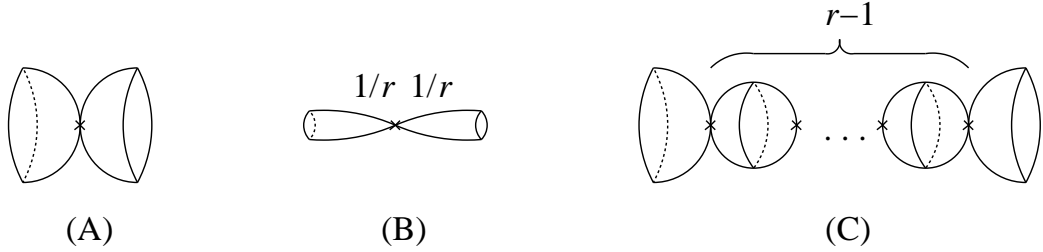}
\caption{(A) a stable curve, (B) an $r$-stable curve, (C)
an $r$-bubbled curve.}
\label{Fig:3versions}
\end{center}
\end{figure}

\paragraph{2) $r$-bubbled spin curves.} The previous
construction was improved in~\cite{CCC}. The idea is to
transform every stable curve with at least one node into
a semistable curve. One obtains
that, on the semistable curves, it is enough to consider
locally free sheaves.

We provide a precise description as follows.
Consider
the resolution of the $A_{r-1}$ singularity at the
origin of $\{xy=t^r\}$ by $\lfloor r/2 \rfloor$ iterated blowups.
One obtains a smooth universal
curve, where every node is replaced by a chain of $r-1$
complex projective lines.

On this universal curve, $\Lcal$ is a line bundle, and we have
a map
$$
\Lcal^{\otimes r} \rightarrow
\omega_{\log}(-\textsum_{i=1}^n m_i[s_i]).
$$
This map is {\em not} an isomorphism. Indeed, its vanishing divisor
is a weighted sum of the projective lines that form the
$(r-1)$-chain, with coefficients
$$
a,\quad 2a, \quad \dots, \quad (b-1)a,\quad
 ba,\quad  b(a-1), \quad \dots, \quad 2b,\quad  b,
$$
for $a,b \in \{1, \dots, r-1\}$ and $a+b=r$.

In this paper, a curve like that will be called
an {\em $r$-bubbled curve}. When endowed with the
line bundle $\Lcal$ as above, it will be called
an {\em $r$-bubbled spin curve}.

\paragraph{3) $r$-stable spin curves.}
Another way to improve the initial construction was given in~\cite{AJ}
and in~\cite{chstab}.
A stack structure at the neighborhood of every node
of every singular curve is introduced. Note that the
surface $xy=t^r$ is the quotient of the smooth surface
$XY=t$ by the action of the group $\ZZ_r = \ZZ/r\ZZ$
via $X \mapsto \xi_r X$, $Y \mapsto \xi_r^{-1} Y$. In
other words, the surface $xy=t^r$ is actually the coarse space
of a smooth stack. Over this smooth stack, $\Lcal$ is a
line bundle and
$$
\Lcal^{\otimes r} \rightarrow
\omega_{\log}(-\textsum_{i=1}^n m_i[s_i])
$$
is a global isomorphism. These stacks are
special cases of Abramovich and Vistoli's
stack-theoretic curves: stacks with stabilizers of arbitrary order
at the nodes and at the markings (they are called ``twisted curves'' in
\cite{AV}).

The fibers of this version of the universal curve are
stable curves endowed with a nontrivial
stack structure at the nodes. They are called {\em $r$-stable
curves} and the {\em $r$-prestable curves} are defined analogously.
The neighborhood of a
node in an $r$-stable curve is isomorphic to an ordinary node
$XY=0$ endowed with the group action of $\ZZ_r$ via
$X \mapsto \xi_r X$, $Y \mapsto \xi_r^{-1} Y$.
When endowed with the line bundle~$\Lcal$ as above, the
curve is called an {\em $r$-(pre)stable spin curve}.

An $r$-stable curve with $k$ nodes has exactly $r^k$
times as many automorphisms as the associated coarse
stable curve~\cite[Thm.~7.1.1]{ACV}.

In this picture, the numbers $a$ and $b$
result from the
 $\ZZ_r$-action involved in the local picture of $\Lcal$
at the two branches.
More precisely, locally on
the $\ZZ_r$-space
$$\{XY=0\} \text{ with $\ZZ_r$-action } (X,Y)\mapsto (\xi_r X,\xi_r^{-1}Y),$$
the total space of $\Lcal$ is the $\ZZ_r$-space
$$\{(X,Y,T)\mid XY=0\} \text{ with $\ZZ_r$-action }(X,Y,T)\mapsto
(\xi_r X, \xi_r^{-1}Y, \xi_r^aT).$$
The index $b$ in $\{1,\dots, r\}$ is determined
by $a+b\equiv 0 \mod r$ or by interchanging
$X$ with $Y$ in the local picture.

\bigskip

\paragraph{On the three constructions.}
Among the three constructions, the third one (involving $r$-stable
curves) is best fit for
constructing the compactification. Indeed,
 $\MMMbar_{g,n}^{r,\pmb m}$ is just the solution of the corresponding
moduli problem. On the other hand, the objects defined in the
first two constructions, carry less information.
In these definitions, it can happen that
two $r$-spin curves lying over two distinct points of
the boundary of $\MMMbar_{g,n}^{r,\pmb m}$
are isomorphic.
Thus in these constructions, the moduli functor
yields a singular stack, which needs to be normalized if we want to
obtain $\MMMbar_{g,n}^{r,\pmb m}$.
(The authors in \cite{Ja_comp} and \cite{CCC} addressed this issue:
the moduli functor of $r$-spin structures should be refined
in an appropriate way \cite{Ja_geom} \cite[(a,b) p.26]{CCC},
but we do not discuss this here.)

On the other hand,
if we wish to apply the Grothendieck--Riemann--Roch (GRR) formula,
it turns out that it is more straightforward
to work with the second construction
involving $r$-bubbled curves. This happens because
the GRR formula applies without
modifications to morphisms between stacks only if the fibers
of the morphisms are schemes, which is not true for $r$-stable curves.

In this paper we will talk about spin structures on
$r$-(pre)stable curves when
describing moduli spaces and morphisms between them,
but we will switch to $r$-bubbled curves when we want to
apply GRR and calculate the Chern character of the
$K$-theoretical direct image of the $r$-spin
structure. The following remark illustrates why this
direct image coincides with the direct image
via the $r$-(pre)stable curve.

\begin{rem}
\label{rem:samech}
There are the following morphisms between universal curves:
$$
p'\colon \CCCbar^{\rm bubble} \rightarrow \CCCbar^{\rm coarse}
\qquad \mbox{and} \qquad
p''\colon \CCCbar^{\rm stacky} \rightarrow \CCCbar^{\rm coarse}.
$$
Moreover, we have $\Lcal^{\rm coarse} = p'_* \Lcal^{\rm bubble} =
p''_* \Lcal^{\rm stacky}.$
\end{rem}

\subsection{Boundary}\label{Ssect:Boundary}

The structure of the boundary of $\MMMbar_{g,n}^{r,\pmb m}$
is similar to that of $\MMMbar_{g,n}$, but there is a new
subtlety.

\paragraph{Notation.} By $[n]$, we denote the set $\{1,\dots, n\}$.
Let $I\subset[n]$. For any multiindex $\pmb m=(m_1, \dots, m_n)$,
we denote by $\pmb m_I$ the multiindex $(m_i)_{i\in I}$.


For any nonnegative integers $n$ and $g$
we define the following involutions.\\
\noindent (1) Given $l\in \{0,\dots,g\}$, we write $l'$ for $g-l$.\\
\noindent (2) Given a
subset $I \subset [n]$ we write $I'$ for $[n] \setminus I$.\\
\noindent (3) Given $q\in \{1,\dots, r\}$ we write $q'$ for the
integer in $\{1,\dots, r\}$ satisfying
$q+q'\in r\ZZ$.

\bigskip

The boundary $\partial \MMMbar_{g,n}^{r,\pmb m}$ of
$\MMMbar_{g,n}^{r,\pmb m}$ is the moduli space of
{\em singular} $r$-stable spin curves.
The normalization $N(\partial \MMMbar_{g,n}^{r,\pmb m})$
of the boundary is the moduli space of pairs
$(C, \mbox{node of } C)$, where $C$ is a singular $r$-stable spin curve.
Finally, we consider a double cover
$\Dcal$ of the normalization, namely, the moduli space
of triples $(C, \mbox{node of } C,
\mbox{branch of } C \mbox{ at the node})$.

While in the case of moduli spaces of stable curves, the
space $\Dcal$ turned out to be a disjoint union
of several smaller moduli spaces, the picture here is more
complicated: the space $\Dcal$ is not isomorphic, but can
be projected to a disjoint union
of smaller moduli spaces.

The stack $\Dcal$ is naturally equipped with two line bundles
whose fibres are the cotangent lines to the chosen branch of the
{\em coarse} stable curve and to the other branch.
We write
$$
\psi,\psi' \in H^2( \Dcal, \QQ)
$$
for their respective
first Chern classes. Note that in this notation, we privilege the
coarse curve, because in this way the classes $\psi$
and $\psi'$ are more easily related to the
classes $\psi_i$ introduced in \S\ref{Sssec:charclasses}.

Recall that the
 spinor bundle $\Lcal$ determines local indices $a$ and $b$ in $\{1,\dots, r\}$
 satisfying $a+b\equiv 0 \mod r$ in one-to-one correspondence with
 the branches of the node. Therefore, to
each point $(C, \mbox{node of } C, \mbox{branch of } C \mbox{ at the
node})$ we associate an index $q\in \{1, \dots, r\}$ by setting either
$q=a$ or $q=b$ depending on the branch.

We can decompose $\Dcal$
according to the topological type of the node
and the index of the spinor bundle~$\Lcal$ at the
chosen branch:
\begin{equation}\label{eq:decompD}
\Dcal = \bigsqcup_{\substack{0\le l \le g\\
I\subseteq [n]}} \Dcal_{l,I}\sqcup \bigsqcup _{1\le q\le r}
\Dcal ^q_{\irr}.
\end{equation}
A point of $\Dcal_{l,I}$ corresponds to a curve with a node
that divides it into a component of genus~$l$ with marking set~$I$
and a component of genus~$l'$ with marking set~$I'$. A point
of $\Dcal^q_{\irr}$ corresponds to a curve with a nonseparating
node, the index of the spinor bundle at the chosen
branch being~$q$. We have the natural morphisms
$j_{l,I}\colon \Dcal_{l, I}\longrightarrow
\MMMbar_{g,n}^{r,\pmb m}$
and $j_{\irr}^q\colon \Dcal_{\irr}^q
\longrightarrow \MMMbar_{g,n}^{r,\pmb m}$.

\begin{rem}
Over $\Dcal_{l,I}$ the
multiplicity index $q\in\{1,\dots, r\}$  is constant and satisfies
\begin{equation}\label{eq:qlI}
2l-1 -\textsum_I (m_i-1)\equiv q \mod r.\end{equation}
We denote this index by $q(l,I)$.
On the other hand, on $\Dcal_{\irr}^q$ the index $q$
is constant by definition.
Let us set $\Dcal_{\irr}=\bigsqcup _q \Dcal_{\irr}^q$.
\end{rem}

There exist natural morphisms from $\Dcal_{l, I}$ and $\Dcal_{\irr}^q$
to the moduli of $r$-stable
spin curves of smaller dimension (see~\cite{JKV}).

\[
\xymatrix{
  &\Dcal_{l,I}\ar[ld]_{\mu_{l,I}}\ar[rd]^{j_{l,I}}&&
l\in \{0, \dots, g\},\ I\subseteq [n],\ q=q(l,I)\\
\MMMbar_{l,\abs{I}+1}^{r,(\pmb m_I,q)}\times
\MMMbar_{l',\abs{I'}+1}^{r,(\pmb m_{I'},q')}&&\MMMbar_{g,n}^{r,\pmb m}&\\
  &\Dcal^q_{\irr}\ar[ld]_{\mu^q_{\irr}}\ar[rd]^{j_{\irr}^q}&&
q\in \{1, \dots, r\}\\
 \MMMbar_{g-1,n+2}^{r,(\pmb m,q,q')}&&
\MMMbar_{g,n}^{r,\pmb m}
}
\]

These morphisms are obtained by
\begin{enumerate}
\item
normalizing the curve at the
node and taking the pullback of the spinor bundle to the
normalization,
\item ``forgetting'' the orbifold structure at
the two new marked points (this is the same as passing to the
coarse space, but only locally), and
\item replacing the spinor
bundle~$\Lcal$ by the sheaf of its invariant sections (this
sheaf turns out to be locally free at the two new marked points).
\end{enumerate}

\begin{figure}[h]
\begin{center}
\
\includegraphics{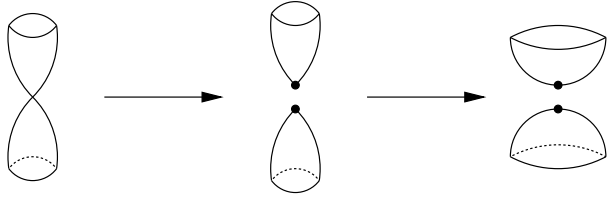}

\caption{The morphisms $\mu$.}
\label{fig:tetons}
\end{center}
\end{figure}

\begin{rem}
When $q=r$,
we actually obtain two $r$-spin structures of types
$(\pmb m_I,0)$ and $(\pmb m_{I'},0)$.
Indeed, $q=r$ is the case where, locally at the node,
the line bundles
are pullbacks of $r$th roots of $\omega_{\log}$ on the coarse space.
If we want the indices $m_i$ to lie in the set $\{1, \dots,r \}$ as usual,
we must further compose the functor $\mu_{l,I}$ with the canonical
isomorphisms recalled in Remark~\ref{rem:shift}
shifting by $r$ the two multiindices
at the two new markings.
\end{rem}

We set $\mu_{\irr}=\bigsqcup _q \mu_{\irr}^q$.

Let us prove that the degree of $\mu_{l,I}$ and $\mu_{\irr}$ equals~1.
Assume, for simplicity that we are in the case where
the generic curve has trivial automorphism group. Set $d = GCD(q,r)$.
Then we claim the following:

{\em A generic geometric point in the image of $\mu_{l,I}$
has a stabilizer of order~$r^2$. It has one geometric preimage
with stabilizer of order~$r^2$. The degree of $\mu_{l,I}$ equals~$1$.

A generic geometric point in the image of $\mu_{\irr}$
has a stabilizer of order~$r$. It has~$d$ geometric preimages
with stabilizers of order~$rd$. The degree of $\mu_{\irr}$ equals~$1$.}

\setlength{\unitlength}{1em}
\begin{center}
\
\begin{picture}(31,5)
\put(0,1){\includegraphics{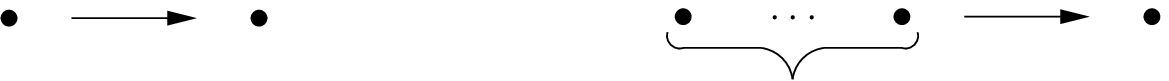}}
\put(-0.1,4){$\frac1{r^2}$}
\put(6.5,4){$\frac1{r^2}$}
\put(17.6,4){$\frac1{rd}$}
\put(23.4,4){$\frac1{rd}$}
\put(30.1,4){$\frac1{r}$}
\put(20.8,0){$d$}
\put(2.7,3.3){$\mu_{l,I}$}
\put(26.1,3.3){$\mu_{\irr}$}
\end{picture}
\end{center}

For both morphisms,
we consider a generic point in the image and work out the fiber over it.

For $\mu_{l,I}$, the point in the image is a pair formed by  a
smooth genus-$l$ $(\abs{I}+1)$-marked curve $C_1$
and a smooth genus-$l'$ $(\abs{I'}+1)$-marked curve $C_2$,
each with trivial automorphism group,
and each equipped with spinor bundles $L_1$ and $L_2$. Then,
due to the ``trivial'' automorphisms acting by multiplication
along the fibres of $L_1$
and---independently---along the fibres of $L_2$,
the point has a stabilizer of order $r^2$.
The geometric points of the fibre are  ``gluings'' of $L_1\to C_1$ and
$L_2\to C_2$; that is to say:
$r$-stable spin curves yielding $L_1\to C_1$ and
$L_2\to C_2$ when we apply the procedures  (1-3) listed above.
There is only one coarse curve obtained from identifying the $(\abs{I}+1)$st
point of $C_1$ to the $(\abs{I'}+1)$st of $C_2$
and, more importantly, there is only one $r$-stable
curve $C$ over it (see for example \cite[Lem.~5.3]{Olsson}).
Furthermore, the fact that the node is separating implies that
there is only one line bundle $L$ up to isomorphism gluing
$L_1$ and $L_2$. Therefore the fibre of $\mu_{l,I}$ contains a single point.
We now show that its stabilizer has order $r^2$ as desired.
As mentioned above, the order of $\Aut(C,\text{markings})$ equals $r$,
because of the presence of a node~\cite[Thm.~7.1.1]{ACV}.
By pulling back the spinor bundle $L$ via any of these
automorphisms, we get another
spinor bundle gluing $L_1$ and $L_2$.
We just showed that, up to isomorphism, there exists only one
such gluing; therefore, we conclude that each automorphism of the
$r$-stable curve lifts to an automorphism of the $r$-stable spin
curve.
Furthermore, the $r$-stable spin curve has
exactly $r$ times as many automorphism as
the $r$-stable curve, because the morphisms lifting the identity of $C$ are
exactly
$r$ ``trivial'' automorphisms acting by
multiplication on the fibres of~$L$.

For $\mu_{\irr}^q$ a generic point in the image is a spinor bundle $L_0$
over a smooth genus-$(g-1)$
$(n+2)$-marked $r$-stable curve $C_0$ satisfying
$\Aut(C_0,\text{markings})=1$. Its automorphism group has order $r$: it
consists of the ``trivial'' automorphisms acting by multiplication
on the fibres of~$L_0$. The geometric points of the fibres are ``gluings'':
$r$-stable spin curves yielding $C_0$ and $L_0$ through the procedures (1-3).
Again, and for the same reasons as above, there is only one $r$-stable curve
$C$ yielding $C_0$ after normalization and passage to the coarse space.
However, since the node is nonseparating, there are exactly
$r$ spinor bundles gluing $L_0$ on the nodal curve $C$.
Let us first show that on $C$
there are exactly $r$ gluings of the trivial bundle on $C_0$:
these are the sheaves $T(0), T(1), \dots, T(r-1)$ of
regular functions $f\colon C_0\to \CC$ satisfying a compatibility conditions
$f(s_{n+1})=\xi_r^i f(s_{n+2})$ at the
$(n+1)$st marking $s_{n+1}$ and at the $(n+2)$nd marking $s_{n+2}$ of $C_0$.
In fact, $\{T(0), T(1), \dots, T(r-1)\}$ is a cyclic group
generated by $T(1)$ and acting freely and transitively on
the gluings of $L_0\longrightarrow C_0$ on $C$;
therefore, once one of such gluings $L$ is fixed,
we can write them all as
$\{L\otimes T(0), L\otimes T(1),\dots, L\otimes T(r-1)\}$.
These line bundles are pairwise non-isomorphic over~$C$, \ie, an
isomorphism
\begin{center}
\setlength{\unitlength}{1em}
\begin{picture}(8,5)
\put(-2,4){$L\otimes T(i)$}
\put(6,4){$L\otimes T(j)$}
\put(3.5,0){$C$}
\put(2.9,4){$\xrightarrow{ \ \sim\ }$}
\put(1,3){\vector(1,-1){2}}
\put(7,3){\vector(-1,-1){2}}
\end{picture}
\end{center}
exists only if $i=j$.

The automorphisms of $(C,\text{markings})$,
form a cyclic group  of order $r$, because of the presence of
a node \cite[Thm.~7.1.1]{ACV}. Denote by~$\alpha$ a generator of this
group. Let $d = GCD(r,q)$.
In \cite[Prop.~2.5.3]{chstab}, the first author shows that there
exists an isomorphism
\begin{center}
\setlength{\unitlength}{1em}
\begin{picture}(8,5)
\put(-2,4){$L\otimes T(i)$}
\put(6,4){$L\otimes T(j)$}
\put(-.5,0){$C$}
\put(7.5,0){$C$}
\put(2.9,4){$\xrightarrow{ \ \sim\ }$}
\put(2.9,0){$\xrightarrow{ \ \sim\ }$}
\put(0,3){\vector(0,-1){1.8}}
\put(8,3){\vector(0,-1){1.8}}
\put(3.5,-.6){$\al$}
\end{picture}
\end{center}
if and only if $j=i+d$.

In this way these automorphisms identify
$r/d$ by $r/d$ the spinor bundles
$$
\{L\otimes T(0), L\otimes T(1),\dots, L\otimes T(r-1)\}.
$$
Furthermore, only the $d$ automorphisms $\al^0, \al^{\frac{r}{d}},
\dots, \al^{(d-1)\frac{r}{d}}$ lift to automorphisms
of the $r$-stable spin curve. As above, due to the ``trivial''
automorphism acting by multiplication on the fibres,
we observe that the automorphisms of the $r$-stable spin curve are
$r$ times as many as the automorphisms lifting from the $r$-stable curve.
In this way the fibre consists of $d$ distinct geometric points
with stabilizers of order $rd$.

\begin{rem}
We point out that $\mu_{l,I}$ and $\mu_{\irr}^q$
are not isomorphisms in general.
The discussion above makes it evident for $\mu_{\irr}^q$: in general,
the generic fibre contains more than one geometric point. A more detailed
analysis shows that also $\mu_{l,I}$ is not an isomorphism in
general\footnote{
We consider for simplicity the case $q(l,I)=r$,
but the discussion extends whenever $q=q(l,I)$ is not prime to $r$.
We claim that
$\mu_{l,I}$ is not fully faithful.
Take an automorphism acting by multiplication by two different factors on
the fibres of  $L_1$ and $L_2$ defined above.
Such an automorphism is not in the image
of the functor $\mu_{l,I}$.
After applying $\mu_{l,I}$ by following the procedures (1-3), we find
that the $r^2$ automorphisms
of the $r$-stable spin curve $L\rightarrow C$
only yield automorphisms acting by multiplication by the same factor
on $L_1$ and $L_2$.
Indeed, the case $q(l,I)=r$ implies that $L$ is a pullback
from the coarse space;
since each automorphism $\al\in \Aut(C, \text{markings})$ fix the coarse space,
the line bundle $\al^*L$ is equal---and not only isomorphic---to $L$.
In this way, all automorphisms lifting $\al$ to $L\to C$ yield the identity
on $L_0\to C_0$.}.
\end{rem}

The discussion above can be extended to the case of stable
maps, as we will see in the next section.

\section{Spaces of spin maps, virtual
classes, and twisted potentials}
\label{Sec:moduli}

The space $X_{g,n,D}^{r, \pmb m} = \MMMbar_{g,n,D}^{r, \pmb m}(X)$
is the {\em moduli space of $r$-stable spin maps of degree $D$}: these
are maps
from $r$-prestable spin curves to the  K\"{a}hler
manifold $X$ whose image cycle is rationally equivalent to $D$.
The stack $X_{g,n,D}^{r, \pmb m}$ is equipped with a universal
$r$-prestable curve,
whose coarse space is the universal coarse prestable curve, in which
$A_{r-1}$ singularities appear.
The desingularization of
this coarse curve is the universal $r$-bubbled curve, which
is well suited for GRR calculations.

\subsection{Moduli of $r$-stable spin maps}\label{Sssect:rstablespinmap}
Let $X$ be a K\"{a}hler manifold. Fix $r\ge 2$, two positive integers
$g$ and $n$, and an effective cycle $D$ in $H_2(X,\ZZ)$.
By $X_{g,n,D}^{r,\pmb m}$, we denote the category of
$r$-stable spin maps. An object is given by the data
\[\xymatrix@R=1cm{
\mathcal L\ar[rd] &&\\
&C\ar[d]^\pi \ar[r]&X&& \fie\colon \mathcal L^{\otimes r}
\xrightarrow{\ \sim \ }\omega_{\log}(-\textsum_{i=1}^n m_i[s_i(B)]), \\
&B\ar@/^1.5cm/[u]^{s_1}_\cdots\ar@/^0.5cm/[u]^{s_n}
}\]
satisfying the following conditions.
\begin{enumerate}
\item $C\to B$ is an $r$-prestable curve.
More explicitly $C$ is a stack of relative dimension one
over a base scheme $B$, its smooth locus is represented by a scheme, and its
singularities are nodes with cyclic stabilizers of order $r$.
The local picture at the nodes
is given by $[\{XY=t\}/\ZZ_r]$, where $t$
is a local parameter on the base scheme
and the group $\ZZ_r$ acts by
$(X,Y,t)\mapsto (\xi_rX , \xi_r^{-1} Y,t)$.
\item the coarse
schemes and morphisms between schemes
corresponding to $C\to B$, $C\to X$, and $s_1, \dots, s_n\colon B\to C$
form a stable map of genus $g$ over $B$, of degree $D$,
and marked at $n$ distinct smooth points.
\item $\mathcal L$ is a line bundle on $C$ and $\fie$ sets an isomorphism
between $\mathcal L^{\otimes r}$ and
$\omega_{\log}(-\textsum_{i=1}^n m_i[s_i(B)])$.
\end{enumerate}
Morphisms are defined in the natural way. Since $C$ is a stack,
in order to obtain a category, we
need to consider $1$-morphisms
up to $2$-isomorphisms (this yields a $2$-category
equivalent to a category as detailed in \cite{AV}).

It follows immediately from \cite{AV}, \cite{Olsson}, and \cite{chstab} that
$X_{g,n,d}^{r,\pmb m}$ is a proper Deligne--Mumford stack.
Beside Kontsevich's construction of
the stack of stable maps, the
key fact is the existence of an algebraic stack of $r$-prestable
curves (a straightforward
consequence of Olsson's
description \cite{Olsson} of the stack of all Abramovich
and Vistoli's stack-theoretic curves).
We detail this as follows.
\begin{enumerate}
\item[(i)]
The functor retaining only the coarse
schemes and morphisms between schemes
corresponding to $C\to B$,  $C\to X$, and $s_1, \dots, s_n\colon B\to C$
lands on Kontsevich's proper stack $X_{g,n,D}$:
\begin{equation}\label{eq:forgetspin}
p\colon X_{g,n,D}^{r,\pmb m}\to X_{g,n,D}.
\end{equation}
\item[(ii)] Note that $X_{g,n,D}$ naturally maps to
$\mathfrak M_{g,n}$, the stack
of genus-$g$ $n$-pointed prestable curves.
The morphism \eqref{eq:forgetspin}, is the base change
via $p\colon X_{g,n,D}\to \mathfrak M_{g,n}$
of the proper morphism
\begin{equation}\label{eq:prestable_spin}
\mathfrak M_{g,n}^{r,\pmb m}\to \mathfrak M_{g,n}
\end{equation}
sending $r$-prestable curves equipped with an $r$-spin structure
$\mathcal L^{\otimes r}\cong\omega_{\log}(-\textsum_{i=1}^n m_i[s_i(B)])$
to the coarse prestable curves. The morphism \eqref{eq:prestable_spin}
is proper and represented by Deligne--Mumford stacks.
This can be regarded as a consequence of
\cite{Olsson}, showing that the stack of $r$-prestable
curves is proper over the stack of
prestable curves, and of \cite{chstab}, showing that the
functor of $r$th roots is proper over the stack of $r$-prestable curves.
\end{enumerate}
\begin{rem}
If $X$ is a point, we recover the stack
$\MMMbar_{g,n}^{r,\pmb m}$ and \eqref{eq:forgetspin}
is the morphism $\MMMbar_{g,n}^{r,\pmb m}\to \MMMbar_{g,n}$.
\end{rem}
\begin{rem}
Moduli of
stable maps equipped with $r$-spin structures were
introduced in \cite{JKVmaps} by means of
Jarvis's notion of coarse stable $r$-spin curve.
As mentioned above, with this technique
several technical points concerning the
singularities appearing in the
moduli stack and the stabilization morphisms
need to be addressed.
The use of Abramovich and Vistoli's stack-theoretic curves
simplifies the treatment of the
moduli stack and avoids dealing with singularities.
This approach to Gromov--Witten $r$-spin theory
via stack-theoretic curves  was  alluded
to in \cite[\S2.3 page 8]{JKV}. After Olsson \cite{Olsson},
this treatment becomes straightforward.
From the point of view of enumerative geometry,
the two approaches of \cite{JKVmaps}
or via stack-theoretic curves are equivalent; this follows immediately
from the case treated in~\cite{AJ},
which can be regarded as the case $X=\pt$ (see also \cite[\S4.3]{ch}).

An irreducible
component of $X^{r,\pmb m}_{g,n,D}$ whose generic points correspond
to singular curves is, in general, nonreduced. Indeed, such a
component is projected onto a ramification locus of the
morphism~\eqref{eq:prestable_spin}. (Another explanation:
following~\cite{AJ} one can realize $X^{r,\pmb m}_{g,n,D}$
as the moduli stack of stable maps to a {\em target stack}.
Such moduli stacks may well be nonreduced.)

\end{rem}

\subsection{The genus-$g$ $r$-spin twisted
Gromov--Witten potential of $X$}
The stack $X_{g,n,D}^{r,\pmb m}$ is equipped with a virtual $r$-spin class,
a homology class
playing the role of the virtual fundamental class in the standard
Gromov--Witten theory of stable maps.
This is the class used in the definition of
intersection numbers in Gromov--Witten $r$-spin theory.
We recall its definition as follows (we refer to
Definitions 4.7 and 5.4 
in \cite{JKVmaps}):
$$\left [X_{g,n,D}^{r,\pmb m}\right ]^v=
\cw(\pmb m)
\cap p^* \left[X_{g,n,D} \right]^v,$$
where $\cw(\pmb m)$ is
the so called Witten's top Chern class,
a cohomology class, whose degree is opposite to the
Euler characteristic of the universal spin structure.
(The explicit expression for the degree
is $((r-2)(g-1)-n+\sum_i m_i)/r$, by Riemann--Roch.)
The class
$\cw$ has several compatible constructions \cite{PV}
\cite{Mo} \cite{chK}, which
extend naturally from $\MMMbar_{g,n}^{r,\pmb m}$ to $X_{g,n,D}^{r,\pmb m}$.
Furthermore, this class
can be understood as the virtual fundamental homology
class of the space of $1/r$-differentials on stable maps, but we
do not develop this point here.

The definitions of the
tautological classes $\psi_1, \dots, \psi_n$,
$\kappa_d$, and $\ch_d$ introduced in \S\ref{Sssec:charclasses}
naturally extend over $X_{g,n,D}^{r, \pmb m}$.

The coefficients appearing in the $r$-spin Gromov--Witten potential
are intersection number of the tautological classes against the
virtual $r$-spin class $\big[ X_{g,n,D}^{r,\pmb m}\big] ^v$.
We recall from \S\ref{defn:twisted_rspin}
the definition of the genus-$g$ $r$-spin twisted Gromov--Witten potential of $X$
$$
F_{g,r}({\pmb t},{\pmb s})=\\
\sum_D Q^D \sum_{n\ge 1}\;\; \frac{1}{n!} \,
\sum_{
\substack{m_1, \dots, m_n \\ a_1, \dots, a_n\\ \mu_1, \dots, \mu_n}
} \; \frac{1}{r^{g-1}}\!\!\!
\int\limits_{\left[ X_{g,n,d}^{r,\pmb m}\right] ^v} \!\!\!\!
\exp\biggl(\sum_{d\ge 1} s_d \, \ch_d \biggr) \;
\prod_{i=1}^n \psi_i^{a_i} \, \ev_i^*(h_{\mu_i}) \;
\prod_{i=1}^n t_{a_i}^{\mu_i\otimes m_i}.
$$
The total $r$-spin twisted Gromov--Witten potential of $X$ is given by
$F_r({\pmb t},{\pmb s})=
\sum\limits_{g\ge 0}\hbar^{g-1}F_{g,r}({\pmb t},{\pmb s})$.

\subsection{Properties of the virtual $r$-spin class}\label{Ssec:propvirt}
We state the properties of the virtual $r$-spin class
$\big[ X_{g,n,d}^{r,\pmb m}\big] ^v$
needed in the rest of the paper.
The main result is Proposition \ref{pro:propvirt}, which can be regarded as the
generalization to $r$-spin maps of
the main factorization properties
of the virtual fundamental class $[X_{g,n,D}]^v$.
The notation follows closely Faber and
Pandharipande's treatment \cite[\S1.2]{FP}.
\begin{rem}In \cite[\S5]{JKV}, the authors provide
a statement of the Gromov--Witten $r$-spin
factorization properties
after pushforward
via the forgetful morphism $p\colon X_{g,n,D}^{r,\pmb m}\to X_{g,n,D}$.
These factorization properties are sufficient if we
wish to intersect $\big[ X_{g,n,D}^{r,\pmb m}\big] ^v$
only with pullbacks from $X_{g,n,D}$.
However,
in the $r$-spin twisted Gromov--Witten potential above, we
consider intersections of classes such as $\ch_d$ that are not
pullbacks from $X_{g,n,D}$.
\end{rem}

To begin with, at (\ref{Ssec:propvirt}.1-3), we set up natural
morphisms $\pi$, $j$, and $\mu$ needed in the statements.
First, recall the morphism forgetting the $(n+1)$st point:
\begin{equation}
\pi\colon X_{g,n+1,D}^{r,(\pmb m, 1)}\longrightarrow X_{g,n,D}^{r,\pmb m}.
\end{equation}
Second, we extend the discussion \S\ref{Ssect:Boundary} of the
boundary locus to
the substack $\partial X_{g,n,D}^{r,\pmb m}
\hookrightarrow X_{g,n,D}^{r,\pmb m}$
of singular objects: i.e. $r$-spin structures over
singular $r$-prestable curves
mapping to $X$.
By the same definitions of \S\ref{Ssect:Boundary}, we get
the locus $\mathcal D$ classifying
triples ($C$, node of $C$, branch of $C$ at the node), which
admits the natural decomposition $$\mathcal D=
\bigsqcup_{\substack{0\le l \le g\\
I\subseteq [n]}} \Dcal_{l,I}\sqcup \bigsqcup _{1\le q\le r}
\Dcal ^q_{\irr}$$ analogue to
\eqref{eq:decompD}. The images of the morphisms
$j_{l,I}\colon \Dcal_{l, I}\longrightarrow
X_{g,n,D}^{r,\pmb m}$
and $j_{\irr}^q\colon \Dcal_{\irr}^q
\longrightarrow X_{g,n,D}^{r,\pmb m}$ describe the entire boundary locus.
As in \S\ref{Ssect:Boundary}, the stack $\mathcal D$
can be projected to certain moduli stacks,
which we denote by $\Delta_\xi$, where $\xi$ labels the topological type
of the splitting of the curve and of the map at the node.

First, we introduce the set $\Omega$ of
such splittings $\xi$
$$
\Omega =
\Omega_{\irr} \sqcup 
\bigsqcup_{\substack{0\le l \le g\\
I\subseteq [n]}} \Omega _{l,I},
$$
with
$\Omega_{l,I}=\{(l,I,A)\mid A\in H_2(X,\ZZ)\},$ \text{and}
$\Omega_{\irr}=\{(\irr,q)\mid q\in \{1,\dots,r\}\}.$
To each $\xi\in \Omega$, we attach a stack $\Delta_\xi$ as follows.
For $\xi=(l,I,A) \in \Omega_{l, I}$, we denote by $\Delta_\xi$
the stack fitting in the fibre diagram
\[\xymatrix@R=0.3cm
{\Delta_\xi\ar[rr]\ar[dd]&&
X_{l',\abs{I'}+1,A'}^{r,(\pmb m_{I'},q')}\ar[dd]^{\ev'}\\
&\square &\\
X_{l,\abs{I}+1, A}^{r, (\pmb m_I, q)}\ar[rr]_{\ev} &&X, }
\]where $A'=D-A\in H_2(X,\ZZ)$, $q$ equals the index $q(l,I)$
defined by equation \eqref{eq:qlI}, and the fibred product
is taken with respect to
the morphisms $\ev$ and $\ev'$ evaluating the $(\abs{I}+1)$st and the
$(\abs{I'}+1)$st point.
Equivalently, we say that
$\Delta_\xi$
fits in the fibre diagram
\[\xymatrix@R=0.3cm
{\Delta_\xi\ar[rr]\ar[dd]&& X\ar[dd]^{\delta=(\id,\id)}\\
&\square &\\
X_{l,\abs{I}+1, A}^{r, (\pmb m_I, q)}\times X_{l',\abs{I'}+1,A'}^{r,(\pmb m_{I'},q')}
\ar[rr]_{\quad \quad \quad \ev\times \ev'} &&X\times X,}
\]
where $\delta\colon X\to X\times X$ is the diagonal morphism.

Similarly, for $\xi=(\irr,q) \in \Omega_{\irr}$, the stack $\Delta_\xi$
fits in the fibre diagram
\[\xymatrix@R=0.3cm
{\Delta_\xi\ar[rr]\ar[dd]&& X\ar[dd]^{\delta=(\id,\id)}\\
&\square &\\
X_{g-1,n+2, D}^{r, (\pmb m, q,q')}\ar[rr]_{\quad (\ev,\ev')} &&X\times X, }
\]
where the morphisms $\ev$ and $\ev'$ evaluating the $(n+1)$st and the
$(n+2)$nd point.

By applying to
$\mathcal D_{l,I}$ and $\mathcal D_{\irr}^q$
the morphism defined in \S\ref{Ssect:Boundary}
(see Figure \ref{fig:tetons}) we obtain the diagrams
\begin{equation}
\xymatrix{
  &\Dcal_{l,I}\ar[ld]_{\mu_{l,I}}\ar[rd]^{j_{l,I}}&&
l\in \{0, \dots, g\},\ I\subseteq [n]\\
\bigsqcup _{\xi\in \Omega_{l,I}} \Delta_\xi &&X_{g,n,D}^{r,\pmb m}&}
\end{equation}
\begin{equation}
\xymatrix{  &\Dcal^q_{\irr}\ar[ld]_{\mu^q_{\irr}}\ar[rd]^{j_{\irr}^q}&&
            q\in \{1, \dots, r\}\qquad \\
\Delta_{(\irr,q)} &&
X_{g,n,D}^{r,\pmb m}&
}
\end{equation}

\begin{rem} We notice that certain stacks
$\Delta_\xi$ may well be empty. This is indeed the case if $A$ or $D-A$ is not
an effective cycle in $H^2(X,\ZZ)$.
\end{rem}

The stacks $\Delta_\xi$ are moduli functors
classifying $r$-stable spin maps and, therefore, are
equipped with virtual $r$-spin classes $\left[ \Delta_\xi\right] ^v$ obtained,
as above, by intersecting
the virtual fundamental class of the corresponding moduli stack of stable maps
with Witten's class $\cw$.
Such virtual $r$-spin classes $\left[\Delta_\xi\right]^v$
can be explicitly obtained as follows
(Axiom 4 of \cite{BM}, see also \cite[\S1.2]{FP}).
For $\xi=(l,I,A)\in \Omega_{l,I}$, we have
\begin{equation}\label{eq:virtlI}
\left[\Delta_{(l,I,A)}\right]^v=
\Big[X_{l,\abs{I}+1,A}^{r,(\pmb m_I,q)}\Big]^v\times
\Big[X_{l',\abs{I'}+1,A'}^{r,(\pmb m_{I'},q')}\Big ]^v\cap
(\ev\times \ev')^{-1}(\delta),\end{equation}
where $\delta$ is the diagonal cycle in $X\times X$. For
$\xi=(\irr,q)\in \Omega_{\irr}$, we have
\begin{equation}\label{eq:virtirr}
\left[\Delta_{(\irr,q)}\right]^v=
\Big[X_{g-1,n+2,D}^{r,(\pmb m,q,q')}\Big]^v\cap
(\ev, \ev')^{-1}(\delta).
\end{equation}

The factorization property  relates
the virtual $r$-spin class of
$\Delta_\xi$ and $X_{g,n,D}^{r,\pmb m}$. Note that it is delicate to
restrict the virtual fundamental class of $X_{g,n,D}^{r,\pmb m}$ to
the boundary locus, because
$X_{g,n,D}^{r,\pmb m}$ may be singular
or simply not of the expected dimension.
This problem already arises for moduli
of stable maps, where calculations involve the natural
morphism from $X_{g,n,D}$ to
the nonsingular algebraic stack $\mathfrak M_{g,n}$ of
prestable $n$-pointed genus-$g$ curves.
Similarly,
we need to regard
$X_{g,n,D}^{r,\pmb m}$ alongside with the natural morphism to
the nonsingular algebraic
stack $\mathfrak M_{g,n}^{r,\pmb m}$ of $r$-prestable curves
equipped with an $r$-spin structure
of type $\pmb m$:
$$X_{g,n,D}^{r,\pmb m}\longrightarrow \mathfrak M_{g,n}^{r,\pmb m}.$$
Then, we take
refined pullbacks $(\mathfrak j_{l,I})^!$ and $(\mathfrak
j_{\irr}^q)^!$  (see~\cite{Fulton}) via the natural morphisms
$$\mathfrak j_{l,I} \colon \mathfrak
D_{l,I}\to \mathfrak M_{g,n}^{r,\pmb m} \qquad \text{and}\qquad
\mathfrak j_{\irr}^q\colon \mathfrak
D^q_{\irr}\to \mathfrak M_{g,n}^{r,\pmb m}$$
induced by the usual decomposition of
the locus $\mathfrak D$ mapping to the boundary
locus of $\mathfrak M_{g,n}^{r,\pmb m}$
representing $r$-spin structures over singular $r$-prestable curves.
Clearly,
the morphisms $\mathfrak j_{l,I}$,
$\mathfrak j_{\irr}^q$, $j_{l,I}$, and $j_{\irr}^q$
 fit in the fibre diagrams
\[\xymatrix@R=0.2cm{
\mathcal D_{l,I}\ar[rr]^{j_{l,I}}\ar[dd]&& X_{g,n,D}^{r,\pmb m}\ar[dd] &&
\mathcal D_{\irr}^q\ar[rr]^{j_{\irr}^q}\ar[dd]&& X_{g,n,D}^{r,\pmb m}\ar[dd]\\
&\square&&&&\square&\\
\mathfrak D_{l,I}\ar[rr]_{\mathfrak j_{l,I}}&& \mathfrak M_{g,n}^{r,\pmb m}.
&& \mathfrak
D_{\irr}^q\ar[rr]_{\mathfrak j_{\irr}^q}&& \mathfrak M_{g,n}^{r,\pmb m}.
}
\]
where  $\mathfrak D_{l,I}$ and
$\mathfrak D ^q_{\irr}$ are the terms of the natural decomposition
 $$\mathfrak D=\bigsqcup_{\substack{0\le l \le g\\
I\subseteq [n]}} \mathfrak D_{l,I}\sqcup \bigsqcup _{1\le q\le r}
\mathfrak D ^q_{\irr}.$$

\begin{pro}\label{pro:propvirt}
The virtual $r$-spin class $\big[ X_{g,n,D}^{r,\pmb m}\big]^v$
satisfies the following properties.

\noindent {\em Factorization property.}
For any $l\in \{0,\dots, g\}$ and $I\subseteq [n]$, we
have
$$(\mu_{l,I})_*(\mathfrak j_{l,I})^!\big[X_{g,n,D}^{r,\pmb m} \big]^v=
\sum_{\xi\in \Omega_{l,I}}
\left[\Delta_{\xi}\right]^v,$$
and for any $q\in \{1,\dots, r\}$, we have
$$(\mu_{\irr}^q)_*(\mathfrak j_{\irr}^q)^!\big[X_{g,n,D}^{r,\pmb m} \big]^v=
\left[\Delta_{(\irr,q)}\right]^v,$$

\noindent {\em Forgetting property.}
We have
\begin{equation*}
\big[X_{g,n+1,D}^{r,(\pmb m,1)} \big]^v=\pi^*
\big[X_{g,n,D}^{r,\pmb m} \big]^v,
\end{equation*}
where
$\pi$ is the morphism
$X_{g,n+1,D}^{r,(\pmb m, 1)}\longrightarrow X_{g,n,D}^{r,\pmb m}$
forgetting the $(n+1)$st point.
\end{pro}

\begin{proof}
The equations above are based on the properties of the
virtual fundamental class $[X_{g,n,D}]^v$ proved in
\cite{BF} \cite{BM} \cite{LT} and on
the properties of Witten's class $\cw(\pmb m)$
proved in \cite{JKV} \cite{PV} \cite{P}.

The forgetting property is an immediate consequence of
the analogous properties for
$[X_{g,n,D}]^v$ and $\cw(\pmb m)$.

On the other hand, the factorization property requires
the Isogeny property of \cite{BM} \cite{BF}. This condition
claims
that for any $l\in \{0,\dots, g\}$ and $I\subseteq [n]$, we
have
$$(\mathfrak j_{l,I})^!p^*\big[X_{g,n,D}
\big]^v=(\mu_{l,I})^*\left(
p^*\big[X_{l,\abs{I}+1,A}\big]^v\times
p^*\big[X_{l',\abs{I'}+1,A'}\big ]^v\cap
(\ev\times \ev')^{-1}(\delta)\right),$$
where, on the right hand side,
$\delta$ is the diagonal in $X\times X$ and we implicitly sum over
the parameter $A$ ranging over $H_2(X,\ZZ)$. Furthermore, for
$q\in \{1, \dots, r\}$, we have
\begin{equation*}
(\mathfrak j_{\irr}^q)^!p^*\big[X_{g,n,D}\big]^v=
(\mu_{\irr}^q)^*\left(p^*\big[X_{g-1,n+2,D}\big]^v\cap
(\ev\times \ev')^{-1}(\delta)\right).
\end{equation*}

Witten's cohomology class $\cw(\pmb m)$
satisfies the following factorization properties
proved in \cite{JKV}, \cite{PV}, and \cite{P}.
For any $l\in \{0, \dots, g\}$ and for any $I\subseteq[n]$, we have
\begin{equation*}
(\mu_{l,I})_* (j_{l,I})^*\left(\cw(\pmb m)\right)=\cw(\pmb m_I, q)
\times \cw(\pmb m_{I'}, {q'}).
\end{equation*}
For any $q\in \{1, \dots, r\}$, we have
\begin{equation*}
(\mu^q_{\irr})_* (j^q_{\irr})^*\left(\cw(\pmb m)\right)=\cw(\pmb m, q, {q'}).
\end{equation*}
The projection formula for $\mu_{l,I}$ and $\mu_{\irr}^q$ yields
immediately the desired factorization
property.
\end{proof}

\begin{rem}
We briefly recall why we restrict the
range of the indices $m_i$ to $\{1,\dots, r\}$.
Indeed, it makes sense to define the moduli of $r$-stable spin curves for any
multiindex $\pmb m\in \ZZ^n$. On the other hand,
the construction of Witten's class $\cw(\pmb m)$
requires that the entries $m_i$ are positive. Furthermore,
this extended Witten's class $\cw(\pmb m)$ satisfies the so called descending and the vanishing properties.
The descending property claims that for all $i\in \{1, \dots, n\}$, we have
\begin{equation*}
\cw(\pmb m+r\pmb \delta_i)=-\frac{m_i}{r}\ \psi_i \ \cw(\pmb m).
\end{equation*}
The vanishing property claims that,
if $m_i\in r\ZZ$ for some $1\le i\le n$, then we have
\begin{equation*}
\cw(\pmb m)=0.
\end{equation*}
Automatically, the
virtual $r$-spin class $\big[ X_{g,n,D}^{r,\pmb m}\big]^v$
shares analogous recursive relations.
Because of these properties, it makes sense to restrict ourselves
to the values $m_i \in \{1, \dots, r-1 \}$.
\end{rem}

\section{Applying the Grothendieck--Riemann--Roch formula}
\label{Sec:GRR}
In~\cite[Thm.~1.1.2]{ch}, the first author applied the
Grothendieck--Riemann--Roch formula to the $r$-spin structure $\Lcal$
over the universal curve. These methods can be easily generalized
to any family of $r$-bubbled spin curves.
In \S\ref{Sssec:GRR},  we describe how the GRR formula works in this case.
In \S\ref{Sssec:GRRonX} we intersect the GRR
formula with the virtual $r$-spin class
of $X_{g,n,D}^{r,\pmb m}$. Finally in \S\ref{Sssec:diffop}
we use this version of the
GRR formula
to deduce the differential equation satisfied
by the twisted $r$-spin potential. The final
result, Proposition \ref{pro:diffop},
is the crucial ingredient allowing us to
generalize Givental's quantization.

\subsection{The GRR formula for $r$-bubbled spin curves}
\label{Sssec:GRR}

In~\cite{ch} the GRR formula was applied to the universal curve
$\pi\colon \CCCbar_{g,n}^{r,\pmb m} \to \MMMbar_{g,n}^{r,\pmb m}$.
In this section we briefly describe these computations to show that
they actually work for any family of $r$-spin curves with maximal
variation.

\begin{defn}
A family of $n$-pointed, $r$-prestable curves $\pi:C \to B$
is a family {\em with maximal variation} if for any $b \in B$
the Kodaira--Spencer homomorphism
$T_bB \to \Ext^1(\Omega_{C_b},\Ocal_{C_b})$ is surjective.
\end{defn}

It follows from the definition, that in a family~$B$ with maximal
variation, the boundary
$\d B = \{b\in B \; | \; C_b \mbox{ is singular} \}$
is a normal crossings divisor in $B$.
As in \S\ref{Ssect:Boundary}, we construct a
smooth scheme $D$ whose points are triples
$(b \in \d B, \mbox{node of } C_b, \mbox{branch at the node})$.
The scheme~$D$ has a decomposition
$D=\bigsqcup_{l,I} D_{l,I}\sqcup \bigsqcup_q D_{\irr}^q$
endowed with morphisms
$j_{l,I}\colon D_{l,I}\to B$ and $j_{\irr}^q\colon D_{\irr}^q\to B.$

The schemes $B$ and $D$ are equipped with the tautological
cohomology classes $\kappa_d, \psi_1,\dots, \psi_n\in H^*(B,\ZZ)$
and $\psi,\psi'\in H^2(D,\ZZ)$ as in~\S\ref{Sssec:charclasses}.

The following result is a slight generalization of the main
theorem of~\cite{ch}.

\begin{pro}[\cite{ch}]\label{pro:GRR}
Consider a family of $r$-prestable curves $C \to B$ with
maximal variation over a nonsingular scheme $B$,
equipped with an $r$-spin structure $\mathcal L$
\[
\xymatrix
{{\mathcal L}\ar[rd] &&\\
&C\ar[d]^\pi & &&\fie\colon {\mathcal L}^{\otimes r}
\xrightarrow{\ \sim \ }\omega_{\log}
(-\textsum_{i=1}^n m_i[s_i(B)]).\\
&B\ar@/^1.5cm/[u]^{s_1}_\cdots\ar@/^0.5cm/[u]^{s_n}
}
\]
Then, we have
\begin{multline*}
\ch_d(R\pi_*\mathcal L) =
\frac{B_{d+1}(\frac1r)}{(d+1)!}\kappa_d
-\sum_{i=1}^n  \frac{B_{d+1}(\frac{m_i}r)}{(d+1)!} \psi_i^d
+\frac{r}{2} \sum_{q=1}^{r}
\frac{B_{d+1}(\frac{q}r)}{(d+1)!} \;
(j^q_{\irr})_* \!
\left(\sum_{a+a'=d-1}  \!\!\!\!\!
(\psi)^a(-\psi')^{a'}\right) \qquad \; \\
+\frac{r}{2} \sum_{\substack{0\le l\le g\\I\subseteq[n]}} \!\!
\frac{B_{d+1}(\frac{q(l,I)}r)}{(d+1)!} \;
(j_{l,I})_* \!
\left(\sum_{a+a'=d-1} \!\!\!\!\!
(\psi)^a(-\psi')^{a'}\right).
\end{multline*}
\end{pro}

\noindent
{\em Sketch of a proof.}
We want to calculate the Chern character of the
direct image $R\pi_*\Lcal$.
The first step of the proof is to replace
the $r$-stable curve $\pi\colon C\to B$
by the corresponding
$r$-bubbled curve $\wt\pi\colon \wt C\to B$.
Replacing $\pi_*$ with $\wt\pi_*$ does not change
the higher direct image because of  the identities of Remark~\ref{rem:samech}.
The GRR formula applied to this situation reads
$$
\ch(R\wt\pi_*\mathcal L)
= \pi_* \bigl(
 \td (\wt \pi) \ch(\Lcal) \bigr).
$$
Now the aim is to compute
the right-hand side.

The maximal variation condition guarantees that the bubbles
form a normal crossings divisor in~$\widetilde C$, while the
nodes of the singular fibers of~$\widetilde C$ form a
smooth subscheme of pure codimension~2.

We can construct $r-1$ families $P_i \to D$, $1 \leq i \leq r-1$,
of projective lines over~$D$
that map to the bubbles of~$\widetilde C$
(starting from the closest bubble to the chosen branch).
In the sequel a {\em bubble} will be the image of one of the
$P_i$ in~$\widetilde C$.
Each family~$P_i$ has two disjoint sections, where
the bubble intersects neighboring bubbles or branches.
The normal sheaves
relative to these sections
are line bundles over~$D$.
The first Chern classes of these
line bundles will be called {\em bubble  
classes}.
They are important for two reasons.
\begin{enumerate}
\item
For a family with smooth fibers, the class $\td(\wt\pi)$ is just the
Todd class of the relative tangent vector bundle; but in our
case not all fibers are smooth, and the nodes give a contribution
to $\td(\wt\pi)$. This contribution is expressed in terms of the
bubble  
classes.
\item
The $r$th tensor power $\Lcal^{\otimes r}$ is isomorphic to the relative
dualizing line bundle~$\omega_{\log}$ to~$\widetilde C$
twisted by the divisors formed by the bubbles
(see \S\ref{Ssec:3curves}). Therefore, to evaluate
$\ch(\Lcal)$ we need to know the intersections between
the classes represented by the bubbles. The intersection
of two different bubbles is simply their geometric intersection
(this is guaranteed by the maximal variation condition). However,
the selfintersection of a bubble is more complicated and
can be described in terms of the bubble  
classes.
\end{enumerate}

It is explained in~\cite{ch} that
the bubble  
classes can be expressed
via $\psi$ and $\psi'$ over every component $D_{l,I}$ and $D^q_{\irr}$.
Once we know this, a computation (though not a simple one)
leads to the result stated in the proposition. \qed

\subsection{The GRR formula and the virtual $r$-spin class}
\label{Sssec:GRRonX}

Proposition \ref{pro:GRR} allows
us to express the homology class $\ch_d\cap \big[ X_{g,n,D}^{r,\pmb m}\big]^v$.
\begin{pro}\label{pro:GRRonX}
We have
\begin{multline*}
\ch_d\cap \big [ X_{g,n,D}^{r,\pmb m}\big ]^v =
\frac{B_{d+1}(\frac1r)}{(d+1)!}\kappa_d \cap
\big [ X_{g,n,D}^{r,\pmb m}\big ]^v
-\sum_{i=1}^n  \frac{B_{d+1}(\frac{m_i}r)}{(d+1)!}
\psi_i^d \cap \big [ X_{g,n,D}^{r,\pmb m}\big ]^v
\\
\qquad +\frac{r}{2} \sum_{q=1}^{r}
\frac{B_{d+1}(\frac{q}r)}{(d+1)!} \;
(j^q_{\irr})_* \!
\left(\sum_{a+a'=d-1}  \!\!\!\!\!
(\psi)^a(-\psi')^{a'}\cap (\mathfrak j_{\irr}^q)^!\big [ X_{g,n,D}^{r,\pmb m} \big ]^v \right) \qquad \; \\
+\frac{r}{2} \sum_{\substack{0\le l\le g\\I\subseteq[n]}} \!\!
\frac{B_{d+1}(\frac{q(l,I)}r)}{(d+1)!} \;
(j_{l,I})_* \!
\left(\sum_{a+a'=d-1} \!\!\!\!\!
(\psi)^a(-\psi')^{a'}\cap (\mathfrak j_{l,I})^!\big [ X_{g,n,D}^{r,\pmb m} \big ]^v \right).
\end{multline*}
\end{pro}
\begin{proof}
In order to apply Proposition \ref{pro:GRR}, we need to show that $X_{g,n,D}^{r,\pmb m}$
can be embedded into a nonsingular Deligne--Mumford stack $\wt B$ equipped with
an $n$-pointed
$r$-prestable
curve $\wt \pi \colon \wt C\to \wt B$ of maximal variation
and with an $r$-spin structure
$\wt {\mathcal L}$ fitting in the following fibre diagram
\[\xymatrix@R=0.1cm{
\mathcal L\ar[dd]\ar[rr]&& \wt {\mathcal L}\ar[dd]\\
&\square&\\
 C\ar[dd]\ar[rr]&& \wt C\ar[dd]\\
&\square&\\
X_{g,n,D}^{r,\pmb m}\ar[rr]&& \wt B,
}
\]
where $\mathcal L\to C$ is the universal $r$-prestable
spin curve over $X_{g,n,D}^{r,\pmb m}$. This is actually the crucial
step of the proof: once $\wt B$ is constructed,
since Proposition \ref{pro:GRR} holds for $R\wt \pi_*\wt {\mathcal L}$
over $\wt B$, the desired
equation follows from the
projection formula for refined intersections \cite[Prop.~8.1.1,(c)]{Fulton}.

First, by \cite[\S1.2, Prop.~1]{FP}, there exists an embedding
of $X_{g,n,D}$ into a nonsingular Deligne--Mumford stack $\wt B'$
and a prestable curve  $\wt C'\to \wt B'$, whose variation is maximal,
and which extends the universal curve of $X_{g,n,D}$.
The markings also extend, and indeed we can
regard $\wt C'\to \wt B'$ as a morphism $\wt B'\to \mathfrak M_{g,n}$ extending $X_{g,n,D}\to \mathfrak M_{g,n}$.
Consider the fibred product $\wt B=\wt B' \times_{\mathfrak M_{g,n}}\mathfrak M_{g,n}^{r,\pmb m}$.
By construction (see (i) and (ii)), the stack $X_{g,n,D}^{r,\pmb m}$ is
the fibre over $\wt B'$ of $\wt B\to \wt B'$; therefore, it is embedded in $\wt B$. Furthermore,
$\wt B$ is naturally equipped with an $r$-prestable curve and
an $r$-spin structure, which extends the universal
$r$-spin structure of
$X_{g,n,D}^{r,\pmb m}$.

In this construction, it is crucial
to notice that $\wt B$ is a nonsingular Deligne--Mumford stack.
This follows from Olsson's description  of the stack of $r$-prestable curves.
Indeed, in \cite{Olsson}, the morphism $\wt B\to \wt B'$ at
a point $x\colon \Spec k\to \wt B'$ is
locally represented by the flat, finite
morphism of nonsingular Deligne--Mumford stacks
\begin{equation}\label{eq:localolsson}
[(\Spec \wt I)/(\pmmu_{r}^m)]\to \Spec I
\end{equation}
where the notation is chosen as follows. The scheme
$\Spec I$ is the versal deformation space of
$x$ at $\wt B'$. The index $m$ equals the number of nodes
of the curve $\wt C'_x$ over $x\in \wt B'$.
The ring $\wt I$ equals $I[z_1, \dots, z_m]/(z_i^{r}
 - t_i, \forall i)$,
where $t_1,\dots, t_m\in I$ are chosen so that $\{t_i=0\}\subset \Spec I$
is the locus where the $i$th
node persists. Finally, the group
$(\pmmu_{r})^m$ acts by multiplication on the coordinates $(z_1,\dots, z_m)$.
\end{proof}

\subsection{The differential operator}\label{Sssec:diffop}
Consider the genus-$g$ $r$-spin twisted Gromov--Witten potential
$F_g(\pmb s, \pmb t)$ of $X$.
Set
$$F_r({\pmb t},{\pmb s})=\sum\limits_{g\ge 0}
{\hbar}
^{g-1}
F_{g,r}({\pmb t},{\pmb s}),
\qquad
\text{and}
\qquad Z_r=\exp(F_r).$$

\begin{pro}\label{pro:diffop} We have
$$\frac{\partial} {\partial s_d} Z_r =
L_{d} Z_r,$$
where $L_d$ is the operator
\begin{multline*}
L_d=
\frac{B_{d+1}(\frac1r)}{(d+1)!}
\frac{\partial}{\partial t_{d+1}^{1\otimes 1}}
-\sum_{\substack{a\ge 0\\ \mu\otimes m}}\frac{B_{d+1}(\frac{m}r)}{(d+1)!}
t_{a}^{\mu\otimes m}\frac{\partial }{\partial t_{a+d}^{\mu\otimes m}}
\\
+ \frac{\hbar}{2}\sum_{\substack{a+a'=d-1\\ \mu, \mu'\\m, m'}}(-1)^{a'}
g^{\mu\otimes m, \mu'\otimes m'}
\frac{B_{d+1}(\frac{m}r)}{(d+1)!}\frac{\partial ^2}{\partial t_{a}^{\mu\otimes m} \partial
t_{a'}^{\mu'\otimes m'}},
\end{multline*}
where all summations are taken over the range $a\ge 0$ and $1\leq m \leq r$.
\end{pro}
\begin{proof}
We can write the statement in terms of $F_r$ using $Z_r=\exp F_r$; we get
\begin{multline}\label{eq:diffonF}
\frac{\partial F_r} {\partial s_d}=
\frac{B_{d+1}(\frac1r)}{(d+1)!}
\frac{\partial F_r}{\partial t_{d+1}^{1\otimes 1}}
-\sum_{\substack{a\ge 0\\ \mu\otimes m}}\frac{B_{d+1}(\frac{m}r)}{(d+1)!}
t_{a}^{\mu\otimes m}\frac{\partial F_r}{\partial t_{a+d}^{\mu\otimes m}}
\\
+ \frac{\hbar}{2}\sum_{\substack{a+a'=d-1\\ \mu, \mu'\\m, m'}}(-1)^{a'}
g^{\mu\otimes m, \mu'\otimes m'}
\frac{B_{d+1}(\frac{m}r)}{(d+1)!}\frac{\partial ^2 F}{\partial t_{a}^{\mu\otimes m} \partial
t_{a'}^{\mu'\otimes m'}} \\
+ \frac{\hbar}{2}\sum_{\substack{a+a'=d-1\\ \mu, \mu'\\m, m'}}(-1)^{a'}
g^{m\otimes\mu, m'\otimes\mu'}\frac{B_{d+1}(\frac{m}r)}{(d+1)!}
\frac{\partial F_{r}}{\partial t_{a}^{\mu\otimes m}}
\frac{\partial F_{r}}{\partial t_{a'}^{\mu'\otimes m'}}.
\end{multline}
Write the right hand side as $R_1+R_2+R_3+
R_4$.
Using  Proposition~\ref{pro:GRRonX}, we decompose also
${\partial F_r}/ {\partial s_d}$
as the sum
of four terms involving
$\kappa_d$, $\psi_i^d$,
classes in the image
of $\bigsqcup_q j_{\irr}^q$, and classes in the image of $\bigsqcup_{l,I} j_{l,I}$.
In each the following four steps we identify these summands.

Step 1: intersection numbers involving $\kappa_d$.
The class $\kappa_d$ can be regarded as
the pushforward of
$\psi_{n+1}^{d+1}$ via $\pi\colon
X_{g,n+1,D}^{r,(\pmb m,1)}\to X^{r,\pmb {m}}_{g,n,D}$.
Then, by the forgetting property, the projection formula yields
$$
\int\limits_{\left[ X_{g,n,D}^{r,\pmb m} \right]^v } \!\!\!\!
\kappa_{d} \;
\prod_{i=1}^n \psi_i^{a_i}\ev_i^*(h_{\mu_i})
 \prod_{j=0}^k\ch_{d_j}
=
\int\limits_{\left[X_{g,n+1,D}^{r,(\pmb m,1)} \right]^v }
\!\!\!\!\!\!\!
\psi_{n+1}^{d+1} \;
\prod_{i=1}^n\psi_i^{a_i}\ev_i^*(h_{\mu_i})\prod_{j=0}^k \ch_{d_j},
$$
where we used the equation
$\psi_{n+1}^{d+1}\pi^*\psi_i=\psi_{n+1}^{d+1}\psi_i$ for all
$d \ge 0$, and $1\le i\le n$.
In this way, we get $R_1$.

Step 2: intersection numbers involving $\psi_i^d$.
These intersections are already in the desired form and yield
immediately $R_2$.

Step 3:
intersection numbers involving classes in the image of $j_{\irr, q}$.
For all $1\le q\le r$, we intersect
$\big(\frac{\hbar}{r}\big)^{g-1}\prod_{i=1}^n
\psi_i^{a_i}\ev_i^*(h_{\mu_i})\prod_{j=0}^k\ch_{d_j}$
with the homology class
$$\frac{r}{2}(j^q_{\irr})_* \!
\left(\sum_{a+a'=d-1}  \!\!\!\!\!
(\psi)^a(-\psi')^{a'}\cap (\mathfrak j_{\irr}^q)^!
\Big [ X_{g,n,D}^{r,\pmb m} \Big ]^v \right)$$
and we multiply by $B_{d+1}(q/r)/(d+1)!$.
Taking aside this last factor, we carry out the intersection
on $\mathcal D_{\irr}^q$, and we get
$$
\frac{\hbar^{g-1}}{2r^{g-2}}
\left[\sum_{a+a'=d-1\!\!\!\!\!\!}(\mu^q_{\irr})^*
\left((\psi_{n+1})^a(-\psi_{n+2})^{a'}\right)
 \cdot
(j_{\irr}^q)^*\left(
\prod_{i=1}^n \psi_i^{a_i}\ev_i^*(h_{\mu_i})
\prod_{j=0}^k\ch_{d_j}\right)\right]\cap
(\mathfrak j_{\irr}^q)^!\Big [ X_{g,n,D}^{r,\pmb m} \Big ]^v
$$
Now instead of integrating via $j_{\irr}^q$ we can integrate
via $\mu^q_{\irr}$.
Notice the identity
$$(j_{\irr}^q)^*\left(\textstyle{\prod_{i=1}^n}
\psi_i^{a_i}\;\ev_i^*(h_{\mu_i})\right)=
(\mu^q_{\irr})^*\left(\textstyle{\prod_{i=1}^n}
\psi_i^{a_i}\;\ev_i^*(h_{\mu_i})\right).$$
As in \cite{JKV}, we also notice that the
terms of the Chern character
in degree $d\ge1$, satisfy the identity
\begin{align*}
&(j_{\irr}^q)^*\ch_d=(\mu^q_{\irr})^*(\ch_d).
\end{align*}
Therefore, the intersection number is
$$
\frac{\hbar^{g-1}}{2r^{g-2}}
\sum_{a+a'=d-1}
(\mu_{\irr}^q)^*\left((\psi_{n+1})^a(-\psi_{n+2})^{a'}
\prod_{i=1}^n \psi_i^{a_i}\;\ev_i^*(h_{\mu_i})
\prod_{j=0}^k\ch_{d_j}\right)\cap
(\mathfrak j_{\irr}^q)^!\left(\Big [ X_{g,n,D}^{r,\pmb m} \Big ]^v\right)
$$
Using the projection formula for $\mu^q_{\irr}$ and the
factorization property  we get
$$
\frac{\hbar^{g-1}}{2r^{g-2}}
\sum_{a+a'=d-1}(-1)^{a'} \!\!\!\!\!\!
\int\limits_{\left[ \Delta_{(\irr ,q)} \right]^v }
 \!\!\!
(\psi_{n+1})^a(\psi_{n+2})^{a'} \;
\prod_{i=1}^n \psi_i^{a_i}\;\ev_i^*(h_{\mu_i})
\prod_{j=0}^k\ch_{d_j}.
$$
Recall that by the K\"{u}nneth formula the Poincar\'{e} dual class
of the diagonal $[\delta]$ in $X\times X$ can be written as
$\sum_{\substack{\mu, \mu'}}g^{\mu,\mu'} h_\mu \times h_{\mu'}$,
where $g^{\mu,\mu'}$ is the inverse matrix of the
Poincar\'{e} pairing matrix $g_{\mu,\mu'}$ on $H^*(X,\QQ)$.
In this way, we obtain the intersection numbers
of $X_{g-1,n+2,D}^{r,(\pmb m,q,q')}$
with an extra $\hbar/2$ factor
$$
\frac{\hbar}{2}\left( \frac{\hbar^{g-2}}{r^{g-2}}
\sum_{\substack{a+a'=d-1\\ \mu, \mu'}}(-1)^{a'}g^{\mu,\mu'} \!\!\!\!\!\!\!\!
\int\limits_{\left[ {X_{g-1,n+2,D}^{r,(\pmb m,q,q')}}\right]^v }
 \!\!\!\!\!\!
(\psi_{n+1})^a\ev_i^*(h_{\mu})(\psi_{n+2})^{a'}\ev_i^*(h_{\mu'}) \;
\prod_{i=1}^n \psi_i^{a_i}\;\ev_i^*(h_{\mu_i})
\prod_{j=0}^k\ch_{d_j}\right),
$$
which, if we sum over $q\in \{1,\dots, r\}$ and take the factor
$B_{d+1}(q/r)/(d+1)!$
into account,
agrees with
$R_3$.

Step 4: intersection numbers involving classes in the image of $j_{l,I}$.
Set $I\subseteq [n]$ and \linebreak $l\in \{0,\dots, g\}$, and write $q$ for $q(l,I)$.
We show that the term involving
$Q^D{t_{a_1}^{\mu_1\otimes m_1}}\dots {t_{a_n}^{\mu_n\otimes m_n}}\hbar^{g-1}$
 in $R_4$ equals
\begin{multline*}
Q^D{t_{a_1}^{\mu_1\otimes m_1}}\dots {t_{a_n}^{\mu_n\otimes m_n}}\left(\frac{\hbar}{r}\right)^{g-1}
\frac{1}{n!}
\prod_{i\in [n]} \psi_i^{a_i}\ev_i^*(h_{\mu_i})
\exp\big(\sum_h s_h \ch_h\big)
\\
\cap
\frac{r}{2}\frac{B_{d+1}\big(\frac{q}{r}\big)}{(d+1)!}(j_{l,I})_*\left(\sum_{a+a'=d-1}\psi^a (-\psi')^{a'}\cap
\mathfrak j_{l,I}^!\left[X_{g,n,D}^{r,\pmb m}\right]\right).
\end{multline*}
As in the previous step we put aside the factor ${B_{d+1}({q}/{r})}/{(d+1)!}$,
and we carry out the intersection on $\Dcal_{l,I}$. We get
\begin{multline}\label{eq:onDlI}
\sum_{a+a'=d-1}(-1)^{a'} Q^D\frac{t_{a_1}^{\mu_1\otimes m_1}\dots
{t_{a_n}^{\mu_n\otimes m_n}}}{n!}\frac{\hbar^{g-1}}{2r^{g-2}}
(j_{l,I})^*\left(\prod_{i\in [n]} \psi_i^{a_i}\ev_i^*(h_{\mu_i})
\exp\big(\sum_h s_h \ch_h\big)\right)
\\(\mu_{l,I})^*\left(
(\psi_{\abs{I}+1})^a\times(\psi_{\abs{I'}+1})^{a'}\right)
\cap
(\mathfrak j_{\irr}^q)^!\Big [ X_{g,n,D}^{r,\pmb m} \Big ]^v,
\end{multline}
where we identified the classes $\psi$ and $\psi'$ on
$\Dcal_{l,I}$ with pullbacks via $\mu_{l,I}$.

Notice that the classes $\psi_i$, $\ev_i^*(h_{\mu_i})$,
and $\ch_d$ satisfy the relations
\begin{align*}
&(j_{l,I})^*\left(\textstyle{\prod_{i=1}^n} \psi_i^{a_i}\ev_i^*(h_{\mu_i})\right)=
(\mu_{l,I})^*\left(\textstyle{\prod_{I}}\psi_i^{a_i}\ev_i^*(h_{\mu_i})\times
\textstyle{\prod_{I'}}\psi_i^{a_i}\ev_i^*(h_{\mu_i})\right),\\
&(j_{l,I})^*\ch_d=(\mu_{l,I})^*((\ch_d\times 1)+ (1\times \ch_d)),&
 \text{for $d\ge 1$.}
\end{align*}
Recall the factorization property of the virtual class from Proposition \ref{pro:propvirt}
$$(\mu_{l,I})_*(\mathfrak j_{l,I})^!\big[X_{g,n,D}^{r,\pmb m} \big]^v=
\sum_{(l,I,A)\in \Omega_{l,I}}\left[\Delta_{(l,I,A)}\right]^v=
\Big[X_{l,b,A}^{r,(\pmb m_I,q)}\Big]^v\times
\Big[X_{l',b',A'}^{r,(\pmb m_{I'},q')}\Big ]^v\cap
(\ev\times \ev')^{-1}(\delta),$$
where $b=\abs{I}+1$ and $b'=\abs{I'}+1$.
These relations, together with formal properties of
the exponent function, allow us to rewrite each summand appearing in
the alternate sum
\eqref{eq:onDlI} in terms  of intersections on the fibred product
$X_{l,b,A}^{r,(\pmb m_I,q)}\times_X
X_{l',b',A'}^{r,(\pmb m_{I'},q')}$:
\begin{multline*}
\frac{\hbar}{2}\sum_{A+A'=D}
\left[Q^A\left(\frac{\hbar}{r}\right)^{l-1}
\prod_{i\in I} \frac{t_{a_i}^{\mu_i\otimes m_i}}{n!}
\left(\psi_{b}^{a}\ev_{b}^*({h_{\mu}})
\prod_{i\in I} \psi_i^{a_i}\ev_i^*(h_{\mu_i})
\exp\big(\sum_h s_h \ch_h\big)\right)
\cap \left[X_{l,b,A}^{r,\pmb m_I,q}\right]^v\right]
\\
\left[Q^{A'}\left(\frac{\hbar}{r}\right)^{l'-1}
\prod_{i\in I'} \frac{t_{a_i}^{\mu_i\otimes m_i}}{n!}
\left(\psi_{b'}^{a'}\ev_{b'}^*({h_{\mu'}})
\prod_{i\in I'} \psi_i^{a_i}\ev_i^*(h_{\mu_i})
\exp\big(\sum_h s_h \ch_h\big)\right)
\cap \left[X_{l',b',A'}^{r,\pmb m_{I'},q'}\right]^v\right],
\end{multline*}
where the K\"{u}nneth formula  $[\delta]=
\sum_{\substack{\mu, \mu'}}g^{\mu,\mu'} h_\mu \times h_{\mu'}$ has been used.
Taking into account the factor ${B_{d+1}({q}/{r})}/{(d+1)!}$,
this yields  the monomial of $R_4$ involving
$Q^D{t_{a_1}^{\mu_1\otimes m_1}}\dots {t_{a_n}^{\mu_n\otimes m_n}}\hbar^{g-1}$.
\end{proof}

\section{Givental's quantization}
\label{Sec:quant}

\setcounter{equation}{0}
\setcounter{subsection}{1}

Now we can prove our main Theorem~\ref{Thm:main}.

We use the notation from Section~\ref{Ssec:mainresult}.
In particular, $H = H^*(X,\QQ) \otimes \QQ^{r-1}$ is a
vector space with a quadratic form~$g$ and $\HH = H((z^{-1}))$
is the corresponding infinite-dimensional symplectic space.

\begin{pro}
The operator $L_d$ from Proposition~\ref{pro:diffop}
is obtained by Weyl quantization rules
from the hamiltonian
$$
P_d = - \sum_{a=0}^\infty \sum_{\mu,m}
\frac{B_{d+1}(\frac{m}r)}{(d+1)!}
q_a^{\mu \otimes m} p_{a+d,\mu \otimes m}
+ \frac12
\sum_{\substack{a+a'=d-1\\ \mu, \mu',m,m'}}
(-1)^d \frac{B_{d+1}(\frac{m}r)}{(d+1)!}
g^{\mu \otimes m, \mu' \otimes m'} p_{a,\mu \otimes m}
p_{a',\mu' \otimes m'}.
$$
on~$\HH$.
\end{pro}

\begin{pro}
The vector field, or the
infinitesimal symplectic transformation, induced by this
hamiltonian is the multiplication by
$$
\frac{z^d}{(d+1)!} \; {\rm id} \otimes \mbox{\rm diag} \left[
B_{d+1} \left( \frac1r \right), \dots, B_{d+1} \left( \frac{r-1}r \right)
\right].
$$
\end{pro}

Both propositions are proved by a simple computation.

Theorem~\ref{Thm:main} follows. \qed

\end{document}